\documentclass[]{amsart}

\def\bk{\mathbb{K}} 
\def\bn{\mathbb{N}}
 
\def\bc{\mathbb{C}} 
\def\bff{\mathbb{F}} 
\def\h{\mathcal{H}}

\newcommand{\bdu}[1]{#1^{\sharp\sharp}}
\newcommand{\com}[1]{#1^{\prime}}

\newcommand{\dnorm}[1]{\Vert#1\Vert_{\rm dec}} 
\newcommand{\cbnorm}[1]{\Vert#1\Vert_{{\rm cb}}} 
 
\newcommand{\sumjn}{\sum_{j=1}^n} 
\newcommand{\sumin}{\sum_{i=1}^n}

\newcommand{\bh}{{\rm B}(\mathcal{H})}

\newcommand{\kh}{{\rm K}(\mathcal{H})}

\newcommand{\cb}[1]{{\rm CB}(#1)}

\newcommand{\matn}[1]{{\rm M}_{n}(#1)}

\newcommand{\normc}[1]{\overline{\overline{#1}}}
\newcommand{\weakc}[1]{\overline{#1}}
\newcommand{\e}[1]{{\rm E}(#1)}
\newcommand{\eb}[1]{{\rm E}_1(#1)}

\newcommand{\peb}[1]{\overline{\rm E_1(#1)}^{p.n.}}
\newcommand{\cbb}[1]{{\rm CB}_1(#1)}

\newcommand{\icb}[1]{\rm ICB{(#1)}}
\newcommand{\icbb}[1]{\rm ICB_1{(#1)}}

\newcommand{\hg}{\stackrel{h}{\otimes}}

\newcommand{\id}[1]{\mathcal{I}_{\rm c}({#1})}

\newcommand{\mtens}{\stackrel{\rm min}{\otimes}}
\newcommand{\Mtens}{\stackrel{\rm max}{\otimes}}

\newtheorem{theorem}{Theorem}[section]

\newtheorem{corollary}[theorem]{Corollary}
\newtheorem{proposition}[theorem]{Proposition}
\theoremstyle{remark}

\theoremstyle{definition}

\newtheorem{example}[theorem]{Example}
\numberwithin{equation}{section}

\begin{document}

\title[]{Pointwise approximation by elementary complete contractions} 
\author{Bojan Magajna} 
\address{Department of Mathematics\\ University of Ljubljana\\
Jadranska 21\\ Ljubljana 1000\\ Slovenia}
\email{Bojan.Magajna@fmf.uni-lj.si}

\thanks{Acknowledgment. The author is grateful to Richard Timoney for his suggestions.}
%\thanks{The author was supported in part by the Ministry of Science
%and Education of Slovenia.}



\begin{abstract}A complete contraction on a C$^*$-algebra $A$, which preserves 
all closed two sided ideals $J$, can be approximated pointwise by elementary complete contractions
if and only if the induced map on  $B\otimes A/J$ is contractive for every C$^*$-algebra
$B$, ideal $J$ in $A$ and C$^*$-tensor norm on $B\otimes A/J$.
A lifting obstruction for such an approximation is also obtained.
\end{abstract}

\maketitle

\section{Introduction and notation}

We often try to  understand the structure of operators and spaces on
which they act in terms of approximation by finite
rank maps. On a C$^*$-algebra $A$, however, we may regard two-sided multiplications $x\mapsto axb$ 
as basic building blocks (instead of rank one operators).
We can try to approximate a more general 
map $\phi$ on $A$, that preserves ideals, by finite sums of two-sided multiplications, that is, by
{\em elementary operators}. {\em Here by an ideal  we will always mean a closed two-sided ideal
and we denote the set of all such ideals in $A$ by $\id{A}$.} 

The class of separable C$^*$-algebras
$A$ on which every (completely) bounded map  preserving ideals  admits  uniform 
approximation by elementary operators  turns out to be very
restrictive: it includes just finite direct sums of C$^*$-algebras of continuous sections
vanishing at $\infty$ of locally trivial C$^*$-bundles of finite type \cite{M5}. On the other hand,
each bounded linear map on an arbitrary C$^*$-algebra $A$, that preserves ideals, can be approximated
pointwise  by elementary operators 
\cite{M1}, \cite{AM}.

In this note we   study the question of 
pointwise approximation by {\em elementary complete contractions} that is, completely contractive
elementary operators. It turns out that the uniform bound on the norms of approximating operators
has a strong effect on the nature of the problem. We  note (although this  will 
not be used here)
that for a complete contraction $\phi$ on a C$^*$-algebra $A$, which preserves ideals,
and a fixed
$x\in A$ there always exists a net of elementary complete contractions $\phi_k$ on $A$
such that $\|\phi(x)-\psi_k(x)\|$ tends to $0$ \cite{M3}. However, given a finite collection 
$x_1,\ldots,x_n$ of elements, there is in general no common
approximating net for all $x_j$ simultaneously. This problem turns out to be closely 
related to the theory of C$^*$-tensor norms.

In case of a von Neumann algebra $R$, it can be deduced from results of Chatterjee and Smith \cite{CS} that
each  complete contraction on $R$, which preserves weak* closed two-sided ideals,
can be approximated by elementary complete contractions
in the point weak* topology if and only if $R$ is injective (at least if $R$ has a separable
predual). When trying to apply this result to study the approximation on general 
C$^*$-algebras (via the universal von Neumann envelopes), we have encountered difficulties  
since C$^*$-algebras are not in general locally reflexive as operator spaces. 

A motivation for this note  
was the desire to  understand the range (and its completions in various topologies)
of the natural map $\mu$ from the Haagerup tensor product $A\hg A$ (or $M(A)\hg M(A)$) 
into $\cb{A}$. This is a complementary problem to the question of when
$\mu$ is isometric, which has been solved  by Somerset
\cite{So} and Archbold, Timoney and Somerset \cite{AST}. Important special cases were studied
earlier by many others (see \cite{S},  \cite{AM} and the references there).

In Section 2 we obtain  a necessary and sufficient condition for a map $\phi$ on a $C^*$-algebra
$A$ to be pointwise approximable by elementary complete contractions. Besides the 
obvious requirement that $\phi(J)\subseteq J$ for each ideal $J$ of $A$, it is also necessary
that the induced map on $B\otimes A/J$ is contractive for every C$^*$-tensor norm on the
algebraic tensor product of $A/J$ with each C$^*$-algebra $B$. We prove
that these conditions are also sufficient, relate them to representations of $A$, derive 
some consequences and present a concrete example. It follows in particular that every
nuclear C$^*$-algebra has the {\em elementary completely contractive approximation property}
(the ECCAP, shortly) in the sense that all  complete contractions, that preserve ideals, can
be pointwise approximated by elementary ones. The author does not know if the converse is
true. 

It is not always easy to verify the conditions of our basic result (Theorem \ref{th21}
below),
since often we do not posses enough knowledge about all possible C$^*$-tensor norms concerned. 
Therefore we derive in Section 3 another necessary condition for approximation by elementary
complete contractions,  based on the possibility of lifting.  As an example, we apply this,
together with Ozawa's characterization of the local lifting property \cite{O}, to show that
the Calkin algebra does not have the ECCAP. 

It is customary to take the coefficients $a_i$ and $b_i$ of an {\em elementary
operator }
$\psi(x)=\sumin a_ixb_i\ \ (x\in A)$
to be in the multiplier algebra $M(A)$ of $A$, but,  when studying the pointwise 
approximation, we
can use an approximate unit of $A$ to modify the coefficients to be in $A$. The set of all elementary 
operators on $A$ (with coefficients in $A$)
is denoted by $\e{A}$ and the unit ball in $\e{A}$ (in the completely bounded norm) by $\eb{A}$. By $\icb{A}$ we shall
denote all completely bounded
maps that preserve all closed ideals in $A$ and by $\icbb{A}$ the unit ball of $\icb{A}$.
The closure of a set $S$ will be denoted by $\normc{S}$ for the norm topology, by
$\overline{S}^{p.n.}$ for the point norm topology and by $\overline{S}$ for the weak*
topology.
 
For  basic facts concerning completely bounded maps,  operator spaces and tensor norms we refer to any 
of the books  \cite{BLM}, \cite{ER}, \cite{P},  \cite{Pi}, \cite{W}. 

\section{Approximation by elementary operators and C$^*$-tensor norms}

\begin{theorem}\label{th21} The following three statements are equivalent for a complete
contraction $\phi$ on a C$^*$-algebra $A$:

(i) $\phi\in\peb{A}$.

(ii) For each representation $\pi:A\to\bh$ and all finite subsets $\{a_1,\ldots,a_n\}$
of $A$ and $\{a_1^{\prime},\ldots,a_n^{\prime}\}$ of $\com{\pi(A)}$ the following inequality
holds:
\begin{equation}\label{21} \|\sum_{i=1}^n\com{a}_i\pi(\phi(a_i))\|\leq\|\sum_{i=1}^n\com{a}_i\pi(a_i)\|.
\end{equation}

(iii) The map $\phi$ preserves all ideals $J\in\id{A}$ and for each 
C$^*$-algebra $B$ and C$^*$-tensor norm $\|\cdot\|$ on 
$B\otimes A/J$ the map $1\otimes\phi_J$ on $B\otimes A/J$ (where $\phi_J$  on
$A/J$ is defined by $\phi_J(a+J)=\phi(a)+J$) is a  contraction.
\end{theorem}

\begin{proof} To prove that (i) implies (ii), given $\phi\in\e{A}$ with $\|\phi\|_{\rm cb}<1$,
and a representation $\pi:A\to\bh$, let $\phi_{\pi}$ be the elementary operator induced by
$\phi$ on the weak* closure $R$ of $\pi(A)$. Note that $\|\phi_{\pi}\|_{\rm cb}=
\|\phi_{\pi}|\pi(A)\|_{\rm cb}$
(a consequence of the Kaplansky density theorem) $\leq\|\phi\|_{\rm cb}$. Since $\|\phi_{\pi}\|_{\rm cb}$ is equal to the
norm of the tensor corresponding to $\phi_{\pi}$ in the central Haagerup tensor product
$R\otimes^h_ZR$ (see \cite{CS}), we have that $\phi_{\pi}$ is of the form
$\phi_{\pi}(x)=\sum_jc_j^*xd_j$ ($x\in R$), where $c_j,\ d_j\in R$ and the columns $c:=(c_1,c_2,\ldots)^T$ and 
$d:=(d_1,d_2,\ldots)^T$ have norms less than $1$. Then we have (with the notation as in (ii))
$$\begin{array}{lllll}\|\sum_i\com{a}_i\pi(\phi(a_i))\|&=&\|\sum_i\com{a}_i\phi_{\pi}(\pi(a_i))\|
&=&\|\sum_i\com{a}_i\sum_jc_j^*\pi(a_i)d_j\|\\
&=&\|\sum_jc_j^*(\sum_i\com{a}_i\pi(a_i))d_j\|&=&\|c^*(\sum_i\com{a}_i\pi(a_i))d\|\\
&\leq&\|c\|\|\sum_i\com{a}_i\pi(a_i)\|\|d\|
&\leq&\|\sum_i\com{a}_i\pi(a_i)\|.
\end{array}$$
This proves (\ref{21}) in the case of $\phi\in\e{A}$ with $\|\phi\|_{\rm cb}<1$, the general case
follows by an approximation.

To prove the implication $(ii)\Rightarrow(iii)$,
first note  that $\phi$ preserves each ideal $J$ of $A$ (take in 
(\ref{21}) $n=1$, $\com{a}_1=1$ and $\pi$ such that $\ker{\pi}=J$). Given a C$^*$-algebra
$B$ and a tensor norm $\|\cdot\|$ on $B\otimes A/J$ as in (iii), we first extend the tensor
norm to $\tilde{B}\otimes\tilde{A}/J$, where  $\tilde{B}$ is
the unitization of $B$ (if $B$ is not unital, and $\tilde{B}=B$ if $B$ is unital,
see \cite[IV.4.3]{T}). Let $\Psi$ be a faithful representation of the completion of
$\tilde{B}\otimes\tilde{A}/J$ on a Hilbert space $\h$, $q:\tilde{A}\to\tilde{A}/J$
the quotient map, and define $\pi:A\to\bh$ by
$$\pi(a)=\Psi(1\otimes q(a)).$$
Given $w=\sum_{i=1}^nb_i\otimes q(a_i)$ in $\tilde{B}\otimes \tilde{A}/J$, we put 
$\com{a}_i=\Psi(b_i\otimes 1)$. Then $\com{a}_i$ commutes with $\Psi(1\otimes A/J)=\pi(A)$.
Using (\ref{21}) and the fact that $\Psi$ is isometric we have that
$$\begin{array}{lll}\|(1\otimes\phi_J)(w)\|&=&
\|\Psi(\sum_{i=1}^nb_i\otimes q(\phi(a_i)))\|\\
&=&\|\sum_{i=1}^n\Psi(b_i\otimes1)\Psi(1\otimes q(\phi(a_i)))\|\\
 &=&\|\sum_{i=1}^n\com{a}_i\pi(\phi(a_i))\|\\
 &\leq&\|\sumin\com{a}_i\pi(a_i)\|\\
  &=&\|\sum_{i=1}^n\Psi(b_i\otimes1)\Psi(1\otimes q(a_i))\|\\
 &=& \|\Psi(\sum_{i=1}^nb_i\otimes q(a_i))\|\\
 &=&\|w\|.
\end{array} $$

Now we will prove the implication $(ii)\Rightarrow(i)$. If $\phi\notin\peb{A}$, then for 
some $n\in\bn$ and $a=(a_1,\ldots,a_n)\in A^n$ the element $y:=(\phi(a_1),\ldots,
\phi(a_n))$ is not contained in the norm closure of the subset 
$$S:=\{(\psi(a_1),\ldots,\psi(a_n)):\ \psi\in \eb{A}\}$$
of $\tilde{A}^n$ (with the max norm). By \cite[1.1]{M4}
there exist a representation $\pi:\tilde{A}\to\bh$ (for some Hilbert space $\h$, so that
$\bh$ is an $\tilde{A}$-bimodule via $\pi$) and a completely bounded 
$\tilde{A}$-bimodule map $\Phi:\tilde{A}^n\to\bh$ such that
\begin{equation}\label{22} \|\Phi(x)\|\leq1\ \mbox{for all}\ x\in S\ \ \mbox{and}\ \ 
\|\Phi(y)\|>1.
\end{equation}
The elements $\com{a}_i:=\Phi(e_i)$, where $e_i=(0,\ldots,1,\ldots,0)$, are in 
$\com{\pi(\tilde{A})}$ since for all $a\in\tilde{A}$
$$\com{a}_i\pi(a)=\Phi(e_i)\pi(a)=\Phi(e_ia)=\Phi(ae_i)=\pi(a)\Phi(e_i)=\pi(a)\com{a}_i.$$
Further, since $\Phi$ is an $A$-bimodule map, $\Phi(x_1,\ldots,x_n)=\sum\com{a}_i\pi(x_i)$
for all $x_i\in A$, and hence (\ref{22}) can be rewritten as 
$$\|\sum_{i=1}^n\com{a}_i\pi(\psi(a_i))\|\leq1\ \mbox{for all}\ \psi\in\eb{A}\ \ 
\mbox{and}\ \ \|\sum_{i=1}^n\com{a}_i\pi(\phi(a_i))\|>1,$$
where we have used the definitions of $S$ and $y$. In particular (taking $\psi$ of the form $\psi(x)=e_kxe_k$, where $(e_k)$ is an approximate
unit in $A$)
\begin{equation}\label{23}\|\sum_{i=1}^n\com{a}_i\pi(\phi(a_i))\|>\|\sum_{i=1}^n
\com{a}_i\pi(a_i)\|.
\end{equation}
This contradicts the assumption  (ii).

It remains to prove that $(iii)\Rightarrow(ii)$. If (ii) does not hold, then there exist
a (non-degenerate) representation $\pi$ of $A$ on a Hilbert space $\h$ and elements $a_j\in A$, $\com{a}_j\in\com{\pi(A)}$ ($j=1,\ldots,n$)
such that (\ref{23}) holds. Set $C=\pi(A)$ and $D=C^*(C\cup\com{C})$, the C$^*$-algebra generated by
$C\cup\com{C}$ inside $\bh$. By (\ref{23}), we may choose an irreducible representation $\Theta:
D\to{\rm B}({\mathcal K})$ such that 
\begin{equation}\label{24} \|\Theta(\sum_{i=1}^n\com{a}_i\pi(\phi(a_i))\|>
\|\Theta(\sum_{i=1}^n\com{a}_i\pi(a_i)\|.
\end{equation}
Since $\Theta(C)\subseteq\Theta(\com{C})^{\prime}$, it follows that 
$$\weakc{\Theta(C)}\cap\com{\Theta(C)}\subseteq\Theta(\com{C})^{\prime}\cap\Theta(C)^{\prime}
=\Theta(D)^{\prime}=\bc1,$$
hence $B:=\Theta(C)^{\prime}$ is a factor. Therefore the linear map $x\otimes y
\stackrel{\mu}{\mapsto} xy$ from
the algebraic tensor product $B\otimes\weakc{\Theta(C)}$ onto the subalgebra of ${\rm B}({\mathcal K})$
generated by $B\cup\weakc{\Theta(C)}$ is an isomorphism \cite[5.5.4]{KR}. Let $J=\ker\pi$. 
Since $A/J\cong\pi(A)=C$, we have  injections
$$B\otimes A/J\to B\otimes C\to B\otimes\weakc{\Theta(C)}\stackrel{\mu}{\to}{\rm B}({\mathcal K}),\ \ b\otimes\dot{a}
\mapsto b\Theta(\pi(a)),$$
which can be used to define a C$^*$-norm on $B\otimes A/J$. Put $b_i=\Theta(\com{a}_i)$.
Since by assumption of (iii) $\phi$ preserves ideals, $\phi$ induces a map $\phi_J$ on
$A/J$ (such that $\phi_J(a+J)=\phi(a)+J$ for all $a\in A$). Then (\ref{24}) says that
\begin{equation}\label{25}\|\sum_{i=1}^nb_i\otimes\phi_J(a_i+J)\|>\|\sum_{i=1}^n
b_i\otimes(a_i+J)\|.
\end{equation}
From $\com{a}_i\in\com{\pi(A)}=\com{C}$ we have  $b_i\in\Theta(\com{C})
\subseteq\com{\Theta(C)}=B$, and by (\ref{25}) the map $1\otimes\phi_J$ on $B\otimes
A/J$ is not contractive. This contradicts the assumption (iii).
\end{proof}

Now we derive some consequences of Theorem \ref{th21}.

\begin{corollary}\label{co22} If the range of  $\phi\in\icbb{A}$ is contained in
a nuclear C$^*$-subalgebra $C$ of $A$, then $\phi\in\peb{A}$. In particular, if $A$ is 
nuclear,  $\icbb{A}=\peb{A}$.
\end{corollary}

\begin{proof} Suppose that $\phi\in\icbb{A}\setminus\peb{A}$. Then by Theorem
\ref{th21} there exist an ideal $J$ in $A$, a C$^*$-algebra $B$, a C$^*$-norm
on $B\otimes A/J$ and elements $a_1,\ldots,a_n$ in $A$ and $b_1,\ldots,b_n$ in $B$
such that (\ref{25}) holds. Since quotients of nuclear C$^*$-algebras are nuclear
(see \cite[9.4.4]{BO} or \cite[XV. 3.4]{T}), $C/(C\cap J)$ is nuclear, hence the C$^*$-tensor
norm on $B\otimes C/(C\cap J)$ that comes from the injection $B\otimes C/(C\cap J)\to
B\otimes A/J$ (where $B\otimes A/J$ carries the norm for which (\ref{25}) holds) must
coincide with the minimal C$^*$-tensor norm on $B\otimes C/(C\cap J)$. Therefore
(\ref{25}) implies that 
$$\|\sumin b_i\otimes(\phi(a_i)+(C\cap J))\|_{B\stackrel{\rm min}{\otimes}C/(C\cap J)}
>\|\sumin b_i\otimes(a_i+J)\|_{B\stackrel{\rm min}{\otimes}A/J}.$$
However, the inclusion $C/(C\cap J)\to A/J$ induces an (isometric) $*$-monomorphism
$B\mtens C/(C\cap J)\to B\mtens A/J$, hence the last inequality can be rewritten as
$$\|\sumin b_i\otimes\phi_J(a_i+J)\|_{B\mtens A/J} > \|\sumin b_i\otimes(a_i+J)\|_{B
\mtens A/J}.$$
But this contradicts the well-known fact that $1\mtens\phi_J$ is a complete contraction
since $\phi_J$ is.
\end{proof}

Recall that a (completely bounded) map $\phi:A\to B$ between C$^*$-algebras is called 
{\em decomposable} if it
can be written as a linear combination of completely positive maps. Equivalently,
$\phi$ is decomposable if and only if there exist completely positive maps $\psi_j:A
\to B$ ($j=1,2$) such that the map
\begin{equation}\label{d}A\to {\rm M}_2(B),\ \ a\mapsto\left[\begin{array}{ll}
\psi_1(a)&\phi(a)\\
\phi(a^*)^*&\psi_2(a)
\end{array}\right]
\end{equation}
is completely positive. The infimum of all $\max\{\cbnorm{\psi_1},\cbnorm{\psi_2}\}$,
where $\psi_1$ and $\psi_2$ range over all completely positive maps from $A$ into
$B$ such that  (\ref{d}) is completely positive, is a norm, denoted by
$\dnorm{\phi}$. For more, see \cite{Ha} or \cite{ER}, \cite{Pi}. 

\begin{corollary}\label{co23} For each ideal $J$ in $A$ denote by $\kappa_J$ the 
canonical embedding of $A/J$ into its second dual $\bdu{(A/J)}=\bdu{A}/\bdu{J}$
(and write $\kappa=\kappa_0$ if $J=0$).

(i) If $\phi\in\peb{A}$, then  $\kappa_J\phi$ is decomposable with
$\dnorm{\kappa_J\phi_J}\leq1$ for all $J\in\id{A}$.

(ii) Conversely, if $\phi\in\icbb{A}$ is the pointwise limit of a net of finite rank
maps $\phi_k\in\icbb{A}$ such that $\dnorm{\kappa\phi_k}\leq1$, 
then $\phi\in\peb{A}$. In particular, if $A$ is simple, all finite rank maps $\phi$ with
$\dnorm{\phi}\leq1$ are in $\peb{A}$.
\end{corollary}

\begin{proof} (i) By Theorem \ref{th21}, for each $J\in\id{A}$ and each C$^*$-algebra
$B$ the map $1\otimes\phi_J$ on $B\Mtens A/J$ is (completely) contractive. Hence it follows
by Kirchberg's Theorem \cite[14.1]{Pi} that $\kappa_J\phi_J$ is decomposable with
$\dnorm{\kappa_J\phi_J}\leq1$. 

(ii) If $\phi\notin\peb{A}$, then by Theorem \ref{21} there exist $J\in\id{A}$, a
C$^*$-algebra $B$ and elements $a_1,\ldots,a_n$ in $A$ and $b_1,\ldots,b_n$ in $B$
satisfying (\ref{25}), which implies that
$$\|\sumin b_i\otimes\phi_J(a_i+J)\|_{B\Mtens A/J} > \|\sumin b_i\otimes(a_i+J)\|_{B
\mtens A/J}.$$
Thus, for sufficiently large $k$ we have 
\begin{equation}\label{26}\|\sumin b_i\otimes\phi_{kJ}(a_i+J)\|_{B\Mtens A/J} >
\|\sumin b_i\otimes(a_i+J)\|_{B\mtens A/J},
\end{equation}
since $\phi_{kJ}(a_i+J)=\phi_k(a_i)+J$ converge to $\phi(a_i)+J=\phi_{J}(a_i+J)$.
However, since $\dnorm{\kappa\phi_k}\leq1$, $1\otimes\phi_k$ is a contraction on 
$B\Mtens A$ by \cite[14.1]{Pi} and, since $\phi_k$ preserves ideals, this induces
a contraction $1\otimes\phi_{kJ}$ on $(B\Mtens A)/(B\Mtens J)\cong B\Mtens(A/J)$. 
Since $\phi_{kJ}$ is of finite rank, this implies that $1\otimes
\phi_{kJ}$ must be contractive as a map from $B\mtens A/J$ to $B\Mtens A/J$ (see 
\cite[11.10]{Pi} or \cite{JL}),  contradicting  (\ref{26}).
\end{proof}

Recall that a C$^*$-algebra $A$ has the weak expectation property (WEP) if the
inclusion $\kappa$ of $A$ into its second dual $\bdu{A}$ has a completely positive and completely
contractive factorization  through some $\bh$ \cite[15.1]{Pi}. An operator space
$X\subseteq\bh$ is {\em $1$-exact} if and only if for every C$^*$-algebra $C$ 
the map $i\otimes 1_C$, where $i:X\to \bh$ is the inclusion, extends to a contraction
$X\mtens C\to\bh\Mtens C$  \cite[17.1]{Pi}.
 
\begin{corollary}\label{co24} Suppose that $A$ is simple and has the WEP and that
$X$ is a $1$-exact operator subspace of $A$. Then every map $\phi\in\cb{A}$ such that
$\kappa\phi:A\to\bdu{A}$ 
is decomposable with $\dnorm{\kappa\phi}\leq1$ can be approximated pointwise on $X$ by
elementary complete contractions from $\e{A}$.
\end{corollary}

\begin{proof} If $\phi|X$ can not be approximated, then for some $n\in\bn$ there exists
$x=(x_1,\ldots,x_n)\in X^n$ such that $y:=(\phi(x_1),\ldots,\phi(x_n))$ is not in the norm
closure of the set $S:=\{(\psi(x_1),\ldots,\psi(x_n)):\ \psi\in\eb{A}\}$. As in the proof
of the implications $(iii)\Rightarrow(ii)\Rightarrow(i)$ in Theorem \ref{th21} it follows that there exists 
a C$^*$-algebra $B$, a C$^*$-norm on $B\otimes A$ and elements $b_1,\ldots,b_n$ in $B$ such
that (\ref{25}) holds (with $J=0$, since $A$ is simple). This implies that 
\begin{equation}\label{27}\|\sumin b_i\otimes\phi(x_i)\|_{B\Mtens A} > \|\sumin b_i\otimes
x_i\|_{B\mtens A}=\|\sumin b_i\otimes x_i\|_{B\mtens X}.
\end{equation}
Let $\bff$ be a free group such that there exists a
$*$-epimorphism from the group C$^*$-algebra $C^*(\bff)$ onto $B$ and denote by
$\iota$ the inclusion of $X$ into $A$. Since $A$ has the WEP, $C^*(\bff)\Mtens A=
C^*(\bff)\mtens A$ \cite[15.5]{Pi}. Since $X$ is $1$-exact, this implies that 
$1\otimes\iota:B\mtens X\to B\Mtens A$ is a (well-defined) contraction (see the proof of the implication
$(i)\Rightarrow (ii)$ in Theorem 17.1 of \cite{Pi} if necessary), hence (\ref{27}) can
be rewritten as 
$$\|\sumin b_i\otimes\phi(x_i)\|_{B\Mtens A} > \|\sumin b_i\otimes x_i\|_{B\Mtens A}.$$
But $1\otimes\phi$ is contractive on $B\Mtens A$ since
$\dnorm{\kappa\phi}\leq1$ \cite[14.1]{Pi}, a contradiction.
\end{proof}

\smallskip
{\em Problem.} Characterize  C$^*$-algebras with $\icbb{A}=\peb{A}$. Are they necessarily
nuclear ?

\smallskip
We have a partial answer to the above question. An operator space 
(in particular a C$^*$-algebra) $X$ is called {\em locally reflexive} if
for each finite dimensional operator space $F$ the equality $\bdu{(X\mtens F)}=
\bdu{X}\mtens F$ holds (completely) isometrically \cite{ER}, \cite{Pi}.

\begin{proposition}\label{pr25} Suppose that $A$ is simple and locally reflexive. Then
the equality $\cbb{A}=\peb{A}$ implies that $A$ is nuclear.
\end{proposition}

\begin{proof} We may assume that $A$ is infinite dimensional, so that  $A$ contains an
element with infinite spectrum \cite[4.6.14]{KR}. Using a sequence of continuous functions
with disjoint supports, we can construct for each $n\in\bn$ complete contractions 
$\iota:\ell_{\infty}(n)\to A$ and $E:A\to \ell_{\infty}(n)$ such that $E\iota=1$ (see 
\cite[2.7]{Ha}). To prove that the W$^*$-envelope $\weakc{A}=\bdu{A}$ of $A$ is injective
(hence $A$ nuclear), it suffices to show that there exists a constant $c$ such that
$\dnorm{\theta}\leq c$ for all complete contractions $\theta:\ell_{\infty}(n)\to\weakc{A}$
and all $n$ \cite[2.1]{Ha}.  Let $n$ be fixed.

Since $A$ is locally reflexive there exists a net $(\theta_k)$ of complete contractions 
from $\ell_{\infty}(n)$ to $A$ such that the net $(\kappa\theta_k)$ converges to $\theta$
in the point weak* topology, where $\kappa:A\to \weakc{A}$ is the canonical inclusion.
Since $\peb{A}=\cbb{A}$ by assumption, for each $k$ there exists a net $(\phi_{kj})_j$ in $\eb{A}$ converging
pointwise to $\theta_kE$. Since $A$ is simple (hence primitive), the natural map 
$A\hg A\to\cb{A}$ is completely isometric \cite[5.4.13]{AM}, hence we have the completely isometric
inclusion $\e{A}\subseteq A\hg A$
and it follows that each $\phi_{kj}\in\eb{A}$ is decomposable
with $\dnorm{\phi_{kj}}=\cbnorm{\phi_{kj}}$. (To see that every $\phi\in\e{A}$ is decomposable, note that 
$\phi=a^*xb$, where $a$ and $b$ are fixed columns with the entries in $A$, and take in (\ref{d}) for $\psi_1$
and $\psi_2$ the maps defined by $\psi_1(x)=a^*xa$ and $\psi_2(x)=b^*xb$.) Thus
$\phi_{kj}=\sum_{\ell=0}^3i^{\ell}\phi_{kj\ell}$, where $\phi_{kj\ell}$ are elementary completely positive complete contractions
on $A$ (and $i=\sqrt{-1}$). If we let, for fixed $k$ and $\ell$,  $\psi_{k\ell}$ be the point 
weak* limit of a suitable subnet of 
$(\kappa\phi_{kj\ell})_j$, the maps $\psi_{k\ell}:A\to\weakc{A}$ are completely positive and
completely contractive and $\sum_{\ell=0}^3i^{\ell}\psi_{k\ell}=\kappa\theta_kE$.
This shows that $\kappa\theta_kE$ is decomposable with $\dnorm{\kappa\theta_kE}\leq4$.
Since $\iota$ is completely positive, it follows that $\kappa\theta_k$ 
($=\kappa\theta_kE\iota$)
is also decomposable with $\dnorm{\kappa\theta_k}\leq4$. Since the net $(\kappa\theta_k)$
converges to $\theta$ in the point weak* topology, it follows that $\theta$ is decomposable
with $\dnorm{\theta}\leq4$
\end{proof}

\begin{example}\label{ex26} Let $A=C^*_{\lambda}(\bff_2)$ be the reduced C$^*$-algebra of the free group
$G=\bff_2$ on two generators $g_1,\ g_2$ and let $\phi$ be the automorphism of $A$ that 
interchanges the generators: $\phi(L_{g_1})=L_{g_2}$, $\phi(L_{g_2})=L_{g_1}$, where for
$g\in G$ we denote by $L_g$ the translation operator $(L_g\xi)(h)=\xi(g^{-1}h)$ ($\xi\in
\ell_2(G)$. Since $A$ is simple \cite{D}, $\phi\in\icbb{A}$.  We will show that $\phi\notin
\peb{A}$. The same argument then shows that the weak* continuous extension $\weakc{\phi}$ of $\phi$ to the
von Neumann algebra $\weakc{A}=W^*(G)$ is not in $\peb{\weakc{A}}$, which is a much
stronger result than the classical fact that the automorphism $\weakc{\phi}$ is not inner
\cite[6.9.43]{KR}.
Recall that the commutant $\com{A}$ of $A$ in ${\rm B}(\ell_2(G))$ is the von Neumann
algebra generated by the right convolvers $R_g$ ($g\in G$), where $(R_g\xi)(h)=\xi(hg)$
($\xi\in\ell_2(G)$). Note that $\phi$ maps the elements $1,L_{g_1},L_{g_1}^*,L_{g_2}$ into
$1,L_{g_2},L_{g_2}^*,L_{g_1}$ (respectively). By Theorem
\ref{th21}, to show that $\phi\notin\peb{A}$,  it suffices to show that
\begin{equation}\label{28}\|1+L_{g_2}R_{g_2}+L_{g_2}^*R_{g_2}^*+L_{g_1}R_{g_1}\| > 
\|1+L_{g_1}R_{g_2}+L_{g_1}^*R_{g_2}^*+L_{g_2}R_{g_1}\|.
\end{equation}
The operator on the left side of (\ref{28}) is a sum of four unitary operators and maps
the unit vector $\delta_e$ ($e$ is the unit of $G$) to $4\delta_e$, hence the left side of
(\ref{28}) is equal to $4$. Suppose (to reach a contradiction) that (\ref{28}) does not
hold. Then, given $\varepsilon>0$, there exists a unit vector $\xi=\sum\xi(g)\delta_g$ in
$\ell_2(G)$ such that
\begin{equation}\label{29}\|(1+L_{g_1}R_{g_2}+L_{g_1}^*R_{g_2}^*+L_{g_2}R_{g_1})\xi\| > 
4-\frac{\varepsilon^2}{4}.
\end{equation}
On the left side of (\ref{29}) we have a sum of four unit vectors in a Hilbert space 
and it follows by an elementary  argument (see e.g. \cite[Exercise 20.2]{Pi})
that 
$$\|L_{g_1}R_{g_2}\xi-\xi\|<\varepsilon,\ \|L_{g_1}^*R_{g_2}^*\xi-\xi\| < \varepsilon\ 
\mbox{and}\ \|L_{g_2}R_{g_1}\xi-\xi\| <\varepsilon.$$
This means that 
\begin{multline}\label{210}\sum_{g\in G}|\xi(g_1^{-1}gg_2)-\xi(g)|^2<\varepsilon^2,\ 
\sum_{g\in G}|\xi(g_1gg_2^{-1})-\xi(g)|^2 <\varepsilon^2\\  
\mbox{and}\ \sum_{g\in G}|\xi(g_2^{-1}gg_1)-\xi(g)|^2<\varepsilon^2.
\end{multline}
Now we perform a similar (but not identical) computation as it is in \cite[IVX.3.9]{T}
(or \cite{KR}). For each subset $T$ of
 $G=\bff_2$ let 
 $$\mu(T)=\sum_{g\in T}|\xi(g)|^2.$$
 Let $S$ be the subset of $\bff_2$ that contains the unit  and all the reduced
 words $w\in\bff_2$ which end with $g_1^j$ with $j\ne0$. Note that $S\cup g_2^{-1}Sg_1=\bff_2$
 and that the sets $S$,
 $g_1^{-1}Sg_2$ and $g_1Sg_2^{-1}$ are disjoint. Therefore
 \begin{equation}\label{211}\mu(S)+\mu(g_2^{-1}Sg_1)\geq1\ \mbox{and}
 \end{equation}
 \begin{equation}\label{212}\mu(S)+\mu(g_1^{-1}Sg_2)+\mu(g_1Sg_2^{-1})\leq1.
 \end{equation}
 Since $|\|x\|-\|y\||\leq\|x-y\|$ for any two vectors $x,y\in\ell_2(S)$, we have
 \begin{equation}\label{213}|\mu(g_1^{-1}Sg_2)^{1/2}-\mu(S)^{1/2}|\leq(\sum_{g\in G}|\xi(g_1^{-1}gg_2)
 -\xi(g)|^2)^{1/2},
 \end{equation}
 hence the first inequality in (\ref{210}) implies that
 $$|\mu(g_1^{-1}Sg_2)-\mu(S)|=|\mu(g_1^{-1}Sg_2)^{1/2}-\mu(S)^{1/2}|
 |\mu(g_1^{-1}Sg_2)^{1/2}+\mu(S)^{1/2}|<2\varepsilon.$$
 Applying the same reasoning to  other two inequalities in (\ref{210}) we
  get
 \begin{equation}\label{214}|\mu(g_1^{-1}Sg_2)-\mu(S)|<2\varepsilon,\ 
 |\mu(g_1Sg_2^{-1})-\mu(S)|<2\varepsilon,\  |\mu(g_2^{-1}Sg_1)-\mu(S)|<2\varepsilon.
 \end{equation}
 The last of these three inequalities together with (\ref{211}) imply that
 $\mu(S)>\frac{1}{2}-\varepsilon$. On the other hand, the first two inequalities in (\ref{214})
 together with (\ref{212}) imply that $3\mu(S)-4\varepsilon<1$, and hence we have
 $(1/2)-\varepsilon<\mu(S)<(1+4\varepsilon)/3$. This is a contradiction 
 if $\varepsilon\leq1/14$.
 \end{example}

\section{Lifting and extensions}

Since it is not always easy to verify the conditions of Theorem \ref{th21}, 
we will present another  obstruction for approximation by 
elementary complete contractions.  As an application, we show that the Calkin algebra 
does not have the ECCAP. This fact seems interesting since Voiculescu's theorem
\cite[Theorem 4]{Ar}  (together with Paulsen's technique of  $2\times2$
matrices) implies that each
complete contraction on $\bh$ ($\h$ separable) which annihilates the ideal $\kh$ of compact
operators can be approximated pointwise by two sided multiplications $x\mapsto axb$ with
$a,b\in\bh$ contractive.

\begin{proposition}\label{pr32}Let $J\in\id{A}$, $B:=A/J$, $q:A\to B$ the quotient map,
$X$ a separable subspace of $A$ and $\psi\in\peb{B}$. Then there exists a sequence $(\phi_k)$
in $\eb{A}$ converging pointwise on $X$ to a complete contraction $\phi:X\to A$ such that
$q\phi=\psi q|X$. That is, the following diagram commutes:
$$\begin{array}{lll}
X&\stackrel{\phi}{\longrightarrow}&A\\
\downarrow q|X& &\downarrow q\\
B&\stackrel{\psi}{\longrightarrow}&B
\end{array}$$
\end{proposition}

\begin{proof} Let ${\rm E}_1^h(B)$ be the set of all $\psi\in\e{B}$ which can be represented
by elements $w\in B\otimes B$ with the Haagerup norm $\|w\|_h<1$. Since the Haagerup norm
dominates the c.b. norm, ${\rm E}_1^h(B)\subseteq\eb{B}$. Note that each $\psi\in {\rm E}_1^h(B)$
can be lifted to an operator $\phi\in\e{A}$ (that is, $q\phi=\psi q$) with $\cbnorm{\phi}<1$. 
Indeed, let $\psi$ be represented by a tensor  $w=\sumjn \dot{a}_j\otimes \dot{b}_j$, where
the row matrix $\dot{a}:=[\dot{a}_1,\ldots,\dot{a}_n]$ and the column matrix 
$\dot{b}:=[\dot{b}_1,\ldots,\dot{b}_n]^T$ (with the entries in $B$) have norms less than $1$.
Regarding $\dot{a}$ and $\dot{b}$ as elements of the C$^*$-algebra $\matn{B}=\matn{A/J}=
\matn{A}/\matn{J}$ (by completing with zero entries), we may lift $\dot{a}$ and $\dot{b}$ to
a suitable row $a$ and a column $b$, both with the entries in $A$, so that the corresponding
elementary operator $\phi(x):=\sumjn a_jxb_j$ on $A$ satisfies $\cbnorm{\phi}\leq\|a\|\|b\|
<1$. Clearly $\phi$ is a lift of $\psi$ (that is, $q\phi=\psi q$).

Now we claim that $\eb{B}\subseteq\overline{{\rm E}_1^h(B)}^{p.n.}$. Assume the contrary, that 
there exists a $\phi\in\eb{B}\setminus\overline{{\rm E}_1^h(B)}^{p.n.}$. Observe that in
the argument of the proof of the implication $(ii)\Rightarrow(i)$ in Theorem \ref{th21} we may 
replace the set $\eb{A}$ with ${\rm E}_1^h(A)$ 
since \cite[1.1]{M4} is still applicable. Then (in our present context
of the C$^*$-algebra $B$ instead of $A$) the argument shows that  there exist a 
representation $\pi$ of $B$ and finite
subsets $\{a_1,\ldots,a_n\}$ of $B$ and $\{\com{a}_1,\ldots,\com{a}_n\}$ of $\com{\pi(B)}$ such
that (\ref{23}) holds. But this is impossible since $\phi\in\eb{B}$.

For a general $\psi\in\peb{B}$, by the previous paragraph there exists a net $(\psi_k)$ in 
${\rm E}_1^h(B)$ converging pointwise to $\psi$. Then the net of maps $\psi_kq|X$ converges 
to $\psi q|X$ pointwise on $X$. By the first paragraph of the proof for each $\psi_k$ 
there exists $\phi_k\in\eb{A}$ such that $q\phi_k=\psi_kq$, hence in particular $q\phi_k|X=\psi_kq|X$.
Since $X$ is separable, it follows by the proof of Arveson 
lifting theorem \cite{Ar} (see also \cite[p. 425]{Pi}) that $\psi q|X$ has a completely
contractive lifting $\phi:X\to A$ such that  
$\phi$ is the pointwise limit of a sequence of operators
in $\eb{A}$ restricted to $X$.
\end{proof}

\begin{example}\label{ex34}Let $\h$ be a separable Hilbert space, $B:=\bh$, $K:=\kh$
the ideal of compact operators in $B$, $C:=B/K$ the Calkin algebra and $q:B\to C$ the
quotient map. Consider a copy $B_0$ of $\ell_{\infty}(n)$ ($n\geq3$) of infinite multiplicity
inside $B$, so that $q$ maps $B_0$ completely isometrically onto a C$^*$-subalgebra
$C_0$ of $C$. Let $E:C\to C_0$ be a completely contractive (completely positive)
projection (which exists since $C_0\cong\ell_{\infty}(n)$). Since by \cite[3.2]{Ha}
$\ell_{\infty}(n)$ 
does not have the $1$-operator  (local) lifting property in the sense of  
\cite[Section 16]{Pi}, by Ozawa's theorem \cite{O},  \cite[16.10]{Pi} there exists a complete contraction 
$\psi_0:C_0\to C$
which does not have any completely contractive lifting $C_0\to B$. 
Set $\psi:=\psi_0E$, a complete contraction on $C$. Then there is no
completely contractive lift $\phi:B_0\to B$ of $\psi_0=\psi|C_0$, otherwise 
$\phi(q|B_0)^{-1}:C_0\to B$ would be a completely contractive lift of $\psi_0$.
It follows now from Proposition \ref{pr32} that $\psi\notin\peb{C}$.
\end{example}

It is not hard to prove that an ideal of a C$^*$-algebra with the  ECCAP also has the ECCAP,
but the author does not know if this property passes to quotients.  Using Proposition
\ref{pr32} we can show that an extension by a nuclear
ideal of a C$^*$-algebra having ECCAP  has the ECCAP.

\begin{proposition}\label{pr33} Let $0\to J\to A\stackrel{q}{\to}B\to 0$ be an exact
sequence of C$^*$-algebras, where $A$ is separable and $J$ nuclear. Then, if $B$ has
the ECCAP, the same holds for $A$.
\end{proposition}

\begin{proof} Let $\phi\in\icbb{A}$ and denote by $\psi$ the map on $B$ induced by $\phi$
so that $\psi q=q\phi$..
By Proposition \ref{pr32} there exists $\phi_0\in\peb{A}$ such that $q\phi_0=\psi q$.
Thus $q(\phi-\phi_0)=0$, hence $(\phi-\phi_0)(A)\subseteq J$. Let $(e_k)_{k\in\bk}$ be a 
quasicentral approximate unit in $J\subseteq A$ (see e. g. \cite{Ar} or \cite{D}). 
Since $e_k\phi(x)\in J$ for all $x\in A$ and $J$ is nuclear, by Corollary
\ref{co22} for each $m\in\bk$ the map $e_m\phi$ is in $\peb{A}$. Thus, both  $\phi_0$ and
$e_m\phi$ are in $\peb{A}$, hence the same must hold for the complete contraction
$$\theta_{km}(x):=\sqrt{1-e_k}\phi_0(x)\sqrt{1-e_k}+\sqrt{e_k}(e_m\phi(x))\sqrt{e_k}\ 
(x\in A)$$
for each $k\in\bk$. Since $\sqrt{e_k}\in J$ and $(e_m)$ is an approximate unit for $J$, 
$$\theta_k(x):=\lim_{m\to\infty}\theta_{km}(x)=\sqrt{1-e_k}\phi_0(x)\sqrt{1-e_k}+
\sqrt{e_k}\phi(x)\sqrt{e_k}\ \mbox{for each}\ x\in A,$$
so that $\theta_k\in\peb{A}$. But we can write 
$$\theta_k(x)=(\sqrt{1-e_k}\phi(x)\sqrt{1-e_k}+\sqrt{e_k}\phi(x)\sqrt{e_k})+\sqrt{1-e_k}
(\phi_0(x)-\phi(x))\sqrt{1-e_k}$$
and, since $(\phi_0-\phi)(x)\in J$ and the approximate unit $(e_k)$ is quasicentral in $A$,
it follows that $\lim_{k\to\infty}\theta_k(x)=\phi(x)$. Thus $\phi\in\peb{A}$.
\end{proof}


\begin{thebibliography}{99}
 
\bibitem{AM} P. Ara and M. Mathieu, {\em Local multipliers of C$^*$-algebras,}
Springer Monographs in Math., Springer-Verlag, Berlin, 2003.
\bibitem{AST} R. J. Archbold, D. W. B. Somerset and R. M. Timoney, {\em On the central
Haagerup tensor product and completely bounded mappings of a C$^*$-algebra}, J. Funct.
Anal. {\bf 226} (2005), 406--428. 
\bibitem{Ar} W. B. Arveson, {\em Notes on extensions of C$^*$-algebras},
Duke. Math. J. {\bf 44} (1977), 329--355.
\bibitem{BLM} D. P. Blecher and C. Le Merdy, {\em Operator algebras and their modules,}
L.M.S. Monographs, New Series {\bf 30}, Clarendon Press, Oxford, 2004.  

\bibitem{BO} N. P. Brown and N. Ozawa, {\em C$^*$-algebras and finite dimensional approximations,}
GSM {\bf 88}, AMS, Providence, RI, 2008.

\bibitem{CS} A. Chatterjee and R. R. Smith, {\em The central Haagerup tensor product and
maps between von Neumann algebras,} J. Funct. Anal. {\bf 112} (1993), 97--120.

\bibitem{D} K. Davidson, {\em C$^*$-algebras by example,} Fields Institute Monographs
{\bf 6}, Amer. Math. Soc., Providence, RI, 1996.


\bibitem{ER} E. G. Effros and Z.-J. Ruan, {\em Operator spaces}, London
Math. Soc. Monographs, New Series {\bf 23}, Oxford University Press,
Oxford, 2000.

\bibitem{Ha} U. Haagerup, {\em Injectivity and decomposition of completely
bounded maps,} pp. 170--222, Lecture Notes in Math. {\bf 1132}, Springer-Verlag,
Berlin, 1985.

\bibitem{JL} M. Junge, C. Le Merdy, {\em Factorization through matrix 
spaces for finite rank operators between $C\sp *$-algebras,} 
Duke Math. J. 100 (1999), 299--319.

\bibitem{KR}  R. V. Kadison and J. R. Ringrose, {\em Fundamentals  of 
the  theory  of operator algebras, Vols. 1, 2},  Academic  Press, 
London, 1983, 1986.
\bibitem{M1} B. Magajna, {\em A transitivity theorem for algebras of elementary 
operators,} Proc. Amer. Math. Soc. {\bf 118} (1993), 119--127.

\bibitem{M3} B. Magajna, {\em A transitivity problem for completely bounded mappings,}
Houston J. Math. {\bf 23} (1997), 109--120.
\bibitem{M4} B. Magajna, {\em C$^*$-convex sets and completely bounded bimodule
homomorphisms,} Proc. Roy. Soc. Edinburgh {\bf 130A} (2000), 375--387.

\bibitem{M5} B. Magajna, {\em Uniform approximation by elementary operators}, 
http://www.fmf.uni-lj.si/~magajna/Publications/Uniform.pdf, to appear in 
Proc. Edinburgh Math. Soc.

\bibitem{O} N. Ozawa, {\em On the lifting property for universal 
$^*$-algebras of operator spaces,} J. Operator Theory {\bf 46} (2001), 579--591. 

\bibitem{P}   V.  I.  Paulsen,  {\em Completely  bounded  maps   and 
operator algebras,} Cambridge Studies in Advanced Mathematics {\bf 78}, Cambridge
University Press, Cambridge, 2002.

\bibitem{Pi} G. Pisier, {\em Introduction to  operator 
space theory}, LMS Lecture Note Series {\bf 294}, Cambridge Univ. Press.,
Cambridge, 2003.

\bibitem{S}  R. R. Smith.  Completely bounded module maps  and  the 
Haagerup tensor product, {\em J. Funct. Anal.} {\bf 102} (1991), 156--175.
\bibitem{So} D. W. B. Somerset, {\em The central Haagerup tensor product of a 
C$^*$-algebra}, J. Operator Theory {\bf 39} (1998), 113--121.
\bibitem{T}  M. Takesaki, {\em Theory of operator algebras I, III} Springer-Verlag,
New-York, 1979, 2001.
\bibitem{W} S. Wassermann, {\em Exact C$^*$-algebras and related topics,}
Lecture Notes Series {\bf 19}, Seoul Nat. Univ., 1994.

\end{thebibliography}
\end{document}